\documentclass[12pt]{amsart}
\usepackage{amssymb}
\usepackage{amsbsy}
\usepackage{amscd}

\usepackage{mathrsfs}
\usepackage[dvipdfm]{graphicx}
\makeatletter
\@ifclassloaded{emsprocart}{
\contact[nakajima@kurims.kyoto-u.ac.jp]{
Hiraku Nakajima,
Research Institute for Mathematical Sciences,
Kyoto University, Kyoto 606-8502,
Japan}
}
{}
\@ifclassloaded{emsprocart}
{\newcommand{\MR}[1]{}}
{\renewcommand{\MR}[1]{}}

\renewcommand{\thesubsection}{\thesection(\@roman\c@subsection)}
\makeatother
\usepackage{version}
\usepackage{color}
\newenvironment{NB}{
\color{red}{\bf NB}. \footnotesize
}{}
\excludeversion{NB}
\newenvironment{NB2}{
\color{blue}{\bf NB}. \footnotesize
}{}
\newtheorem{Theorem}[equation]{Theorem}

\newtheorem{Lemma}[equation]{Lemma}

\theoremstyle{definition}

\newtheorem{Conjecture}[equation]{Conjecture}

\theoremstyle{remark}

\numberwithin{equation}{section}
\newcommand{\thmref}[1]{Theorem~\ref{#1}}

\newcommand{\subsecref}[1]{\S\ref{#1}}

\newcommand{\lsp}[2]{{\mskip-.3mu}{}^{#1}\mskip-1mu{#2}}
\newcommand{\defeq}{\overset{\operatorname{\scriptstyle def.}}{=}}
\newcommand{\C}{{\mathbb C}}
\newcommand{\Z}{{\mathbb Z}}
\newcommand{\Q}{{\mathbb Q}}
\newcommand{\R}{{\mathbb R}}
\newcommand{\GL}{\operatorname{GL}}

\newcommand{\gl}{\operatorname{\mathfrak{gl}}}

\newcommand{\g}{{\mathfrak g}}

\newcommand{\Hom}{\operatorname{Hom}}
\newcommand{\Ext}{\operatorname{Ext}}

\newcommand{\Ima}{\operatorname{Im}}

\newcommand{\id}{\operatorname{id}}

\newcommand{\vin}[1]{\operatorname{i}(#1)} % incoming vertex
\newcommand{\vout}[1]{\operatorname{o}(#1)} % outgoing vertex
\newcommand{\bc}{\mathbf c}
\newcommand{\gr}{\operatorname{gr}} % grade
\newcommand{\Uq}{{\mathbf U}_q} % the QUE algebra
\newcommand{\Un}{{\mathbf U}_q^-}
\newcommand{\Uni}{{}_\A\!{\mathbf U}_q^-}
\newcommand{\HomE}{\mathbf E}

\newcommand{\A}{\mathbf A}
\newcommand{\B}{\mathscr B}
\newcommand{\LL}{{\mathscr L}} % the crystal lattice

\newcommand{\Irr}{\operatorname{Irr}}
\newcommand{\Res}{\operatorname{Res}}

\newcommand{\wt}{\operatorname{wt}}

\newcommand{\up}{{\operatorname{up}}}
\newcommand{\Cons}{C}

\newcommand{\scr}{\mathscr}

\newcommand{\te}{\tilde e}
\newcommand{\tf}{\tilde f}

\usepackage[bookmarks=false]{hyperref}

\usepackage{url}

\makeatletter

\@ifclassloaded{emsprocart}{
\newcommand{\subjclass}[2][2000]{}
}
{
\renewenvironment{keywords}{}{}
\excludeversion{classification}
\excludeversion{keywords}
}

\@ifclassloaded{emsprocart}{
\title[Cluster algebras and perverse sheaves]
{Cluster algebras and singular supports of perverse sheaves
}
\author[Hiraku Nakajima]{Hiraku Nakajima\thanks{Supported by the Grant-in-aid
for Scientific Research (No.23340005), JSPS, Japan.
}}
\begin{document}
}
{
\begin{document}
\title[Cluster algebras and perverse sheaves]
{Cluster algebras and singular supports of perverse sheaves
%\\ {\rm \small (Preliminary version: \today)}
}
\author{Hiraku Nakajima}
\address{Research Institute for Mathematical Sciences,
Kyoto University, Kyoto 606-8502,
Japan}
\email{nakajima@kurims.kyoto-u.ac.jp}
}

\begin{abstract}
  We propose an approach to Geiss-Leclerc-Schroer's conjecture on the
  cluster algebra structure on the coordinate ring of a unipotent
  subgroup and the dual canonical base. It is based on singular
  supports of perverse sheaves on the space of representations of a
  quiver, which give the canonical base.
\end{abstract}

\begin{classification}
Primary 13F60; Secondary 17B37, 35A27.
\end{classification}

\subjclass[2000]{Primary 13F60; Secondary 17B37, 35A27}

\begin{keywords}
Cluster algebras, perverse sheaves, singular supports
\end{keywords}

\makeatother

\maketitle

\section*{Introduction}

In \cite{cluster}, the author found an approach to the theory of
cluster algebras, based on perverse sheaves on graded quiver
varieties. This approach gave a link between two categorical
frameworks for cluster algebras, the additive one via the cluster
category by Buan et al.\ \cite{BMRRT} and the multiplicative one via
the category of representations of a quantum affine algebra
by Hernandez-Leclerc \cite{HerLec}.
See also the survey article \cite{Leclerc}.

In \cite[\S1.5]{cluster}, the author asked four problems in the to-do
list, to which he thought that the same approach can be
applied. Except the problem (3), they have been subsequently solved in
works by Qin \cite{Qin,Qin2} and Kimura-Qin \cite{KQ}.

Let us explain other related problems, for which the approach does not
work.
We need a {\it new\/} idea to attack these.

A problem, discussed in this paper, is a natural generalization of the
problem (4) in the to-do list. The author asked to find a relation
between the work of Geiss-Leclerc-Schr\"oer \cite{GLSc} and
\cite{cluster} there. But the work \cite{GLSc} dealt with more general
cases than those corresponding to \cite{cluster}.
A main conjecture says that every cluster monomial in the coordinate
ring of a unipotent subgroup is a Lusztig's dual canonical base element.
Later the theory is generalized to the $q$-analog, where the canonical
base naturally lives \cite{GLSq}.
Therefore it is desirable to find a relation between the cluster
algebra structure and perverse sheaves on the space of quiver
representations, which give the canonical base of the quantum
enveloping algebra.

Another problem is not discussed here, but possibly related to the
current one via \cite{HerLec2}.
In \cite{HerLec}, Hernandez-Leclerc conjectured that the Grothendieck
ring of representations of the quantum affine algebra has a structure
of a cluster algebra so that every cluster monomial is a class of an
irreducible representation.
Those irreducible representations are simple perverse sheaves on
graded quiver varieties.
But what was proved in \cite{cluster} is the special case of the
conjecture only for a certain subalgebra of the Grothendieck ring.
The first part of the conjecture has been subsequently proved in
\cite{HerLec2}. But the latter part, every cluster monomial is an
irreducible representation, is still open.

In this paper, we propose an approach to the first problem. It is not
fully developed yet. We will give a few results, which indicate that
we are going in the right direction. The new idea is to use the singular
support of a perverse sheaf, which is a lagrangian subvariety in the
cotangent bundle. The latter is related to the representation theory
of the preprojective algebra, which underlies the work \cite{GLSc}.

The paper is organized as follows. In the first section, we briefly
recall the canonical and semicanonical bases. 
In the second section, we review works of Geiss-Leclerc-Schr\"oer
\cite{GLSc,GLSq}, where a quantum cluster algebra structure on a
quantum unipotent subgroup is introduced. A subcategory, denoted by
$\mathcal C_w$, of the category of nilpotent representations of the
preprojective algebra plays a crucial role.
In the third section, we study how the singular support behaves under
the restriction functor for perverse sheaves. The restriction functor
gives a multiplication in the dual of the quantum enveloping algebra.
Our main result is the estimate in \thmref{thm:est}.
In the final section, we give two conjectures, which give links
between the theory of \cite{GLSc,GLSq} and perverse sheaves via
singular support.

\subsection*{Acknowledgments}

The main conjecture (Conj.~\ref{conj:2}) was found in the spring of
2011, and has been mentioned to various people since then.
The author thanks Masaki Kashiwara and Yoshihisa Saito for discussion
on the conjecture.
He also thanks the referee who points out a relation between the main
conjecture and a conjecture in \cite[\S1.5]{GLS1}.

\section{Preliminaries}

\subsection{Quantum enveloping algebra}

Let $\g$ be a symmetrizable Kac-Moody Lie algebra. 
We assume $\g$ is symmetric, as we use an approach to $\g$ via the
Ringel-Hall algebra for a quiver.
Let $I$ be the index set of simple roots, $P$ be the weight lattice,
and $P^*$ be its dual. Let $\alpha_i$ denote the $i^{\mathrm{th}}$
simple root.

Let $\Uq$ be the corresponding quantum enveloping algebra, that
is a $\Q(q)$-algebra generated by $e_i$, $f_i$ ($i\in I$), $q^h$
($h\in P^*$) with certain relations. Let $\Un$ be the subalgebra
generated by $f_i$.
We set $\wt(e_i) = \alpha_i$, $\wt(f_i) = - \alpha_i$, $\wt(q^h) = 0$.
Then $\Uq$ is graded by $P$.

The quantum enveloping algebra $\Uq$ is a Hopf algebra. We have a
coproduct $\Delta\colon\Uq\to \Uq\otimes\Uq$. It does not preserve
$\Un$, but Lusztig introduced its modification
$r\colon \Un\to\Un\otimes\Un$ such that
$r(f_i) = f_i\otimes 1 + 1\otimes f_i$ and $r$ is an algebra homomorphism
with respect to the multiplication on $\Un\otimes\Un$ given by
\begin{equation}\label{eq:tm}
  (x_1\otimes y_1)\cdot (x_2\otimes y_2) = q^{-(\wt x_2,\wt y_1)}
  x_1x_2\otimes y_1 y_2,
\end{equation}
where $x_i$, $y_i$ are homogeneous elements. We call $r$ the {\it
  twisted coproduct}.

Let $\A = \mathbb Z[q,q^{-1}]$. Then $\Un$ has an $\A$-subalgebra
$\Uni$ generated by $q$-divided powers $f_i^{(n)} = f_i^n/[n] !$,
where $[n] = (q^n - q^{-n})/(q-q^{-1})$ and $[n]! = [n][n-1]\cdots
[1]$. Then $r$ induces $\Uni\to \Uni\otimes\Uni$, which is denoted
also by $r$.

\subsection{Perverse sheaves on the space of quiver representations
  and the canonical base}

Consider the Dynkin diagram $\mathcal G = (I,E)$ for the Kac-Moody Lie
algebra $\g$, where $I$ is the set of vertices, and $E$ the set of
edges. Note that $\mathcal G$ does not have an edge loop, i.e., an
edge connecting a vertex to itself.

Let $H$ be the set of pairs consisting of an edge together with its
orientation. So we have $\# H = 2\# E$.
For $h\in H$, we denote by $\vin{h}$ (resp.\ $\vout{h}$) the incoming
(resp.\ outgoing) vertex of $h$.  For $h\in H$ we denote by $\overline
h$ the same edge as $h$ with the reverse orientation.
Choose and fix an orientation $\Omega$ of the graph,
i.e., a subset $\Omega\subset H$ such that
$\overline\Omega\cup\Omega = H$, $\Omega\cap\overline\Omega = \emptyset$.
The pair $(I,\Omega)$ is called a {\it quiver}.

Let $V = (V_i)_{i\in I}$ be a finite dimensional $I$-graded vector
space over $\C$. The dimension of $V$ is a vector
\[
  \dim V = (\dim V_i)_{i\in I}\in \mathbb Z_{\ge 0}^I.
\]
%We denote the $i^{\mathrm{th}}$ coordinate vector by $\be_i$.

We define a vector space by
\begin{equation*}
  \HomE_V \defeq \bigoplus_{h\in \Omega} \Hom(V_{\vout{h}}, V_{\vin{h}}).
\end{equation*}
%A point $B$ in $\HomE_V$ is called a {\it representation of a quiver}.

Let $G_V$ be an algebraic group defined by
\begin{equation*}
   G_V \defeq \prod_i \GL(V_i).
\end{equation*}
Its Lie algebra is the direct sum $\bigoplus_i \gl(V_i)$.
The group $G_V$ acts on $\HomE_V$ by
\begin{equation*}\label{eq:Kaction}
  B = (B_h)_{h\in\Omega} \mapsto g\cdot B = 
  (g_{\vin{h}} B_h g_{\vout{h}}^{-1})_{h\in\Omega}.
\end{equation*}
%When $B$ and $B'$ are in the same orbit, we say $B$ and $B'$ are {\it
%  isomorphic\/} as quiver representations.

The space $\HomE_V$ parametrizes isomorphism classes of
representations of the quiver with the dimension vector $\dim V$
together with a linear base of the underlying vector space compatible
with the $I$-grading. The action of the group $G_V$ is induced by the
change of bases.

In \cite{Lu-can2,Lu-book} Lusztig introduced a full subcategory
$\mathcal P_V$ of the abelian category of perverse sheaves on
$\HomE_V$. Its definition is not recalled here. See
\cite[\S2]{Lu-can2} or \cite[Chap.~9]{Lu-book}. Its objects are
$G_V$-equivariant.

Let $\scr D(\HomE_V)$ be the bounded derived category of complexes of
sheaves of $\C$-vector spaces over $\HomE_V$.
Let $\mathcal Q_V$ be the full subcategory of $\scr D(\HomE_V)$
consisting of complexes that are isomorphic to finite direct sums of
complexes of the form $L[d]$ for $L\in\mathcal P_V$, $d\in\Z$.

Let $\scr K(\mathcal Q_V)$ be the Grothendieck group of $\mathcal
Q_V$, that is the abelian group with generators $(L)$ for isomorphism
classes of objects $L$ of $\mathcal Q_V$ with relations $(L) + (L') =
(L'')$ whenever $L''$ is isomorphic to $L\oplus L'$. It is a module
over $\A = \mathbb Z[q,q^{-1}]$, where $q$ corresponds to the
shift of complexes in $\mathcal Q_V$. Then $\scr K(\mathcal Q_V)$ is a
free $\A$-module with a basis $(L)$ where $L$ runs over $\mathcal
P_V$.

Let us consider the direct sum $\bigoplus_V \scr K(\mathcal Q_V)$ over
all isomorphism classes of finite dimensional $I$-graded vector
spaces.
Let $S_i$ be the $I$-graded vector space with $\dim S_i = 1$, $\dim
S_j = 0$ for $j\neq i$. Then the corresponding space $\HomE_{S_i}$ is
a single point. Let $1_i$ be the constant sheaf on $\HomE_{S_i}$,
viewed as an element in $\scr K(\mathcal Q_{S_i})$.
Then Lusztig defined a multiplication and a twisted coproduct on
$\bigoplus_V \scr K(\mathcal Q_V)$ such that the $\A$-algebra
homomorphism
\begin{equation*}
  \Phi\colon \Uni \to \bigoplus_V \scr K(\mathcal Q_V)
\end{equation*}
with $\Phi(f_i) = 1_i$ is an isomorphism respecting twisted coproducts
\cite{Lu-can2,Lu-book}. The construction was motivated by an earlier
work by Ringel \cite{Ringel}.

We here recall the definition of the twisted coproduct on $\bigoplus_V
\scr K(\mathcal Q_V)$. For the definition of the multiplication, see
the original papers.

Let $W$ be an $I$-graded subspace of $V$. Let $T = V/W$.
Let $E(W)$ be the subspace of $\HomE_V$ consisting of $B\in\HomE_V$
which preserves $W$.
We consider the diagram
\begin{equation}
%   \begin{CD}
%     E(W) @>i>> \HomE_V
% \\
%   @V{p}VV @.
% \\
%   \HomE_T\times \HomE_W @.
%   \end{CD}
  \HomE_T\times \HomE_W \xleftarrow{\kappa} E(W) \xrightarrow{\iota}
  \HomE_V,
\end{equation}
where $\iota$ is the inclusion and $\kappa$ is the map given by
assigning to $B\in E(W)$, its restriction to $W$ and the induced map
on $T$.

Consider the functor
\begin{equation*}
  \Res \defeq \kappa_!\iota^* (\bullet) [d]
  \colon \scr D(\HomE_W) \to \scr D(\HomE_T\times\HomE_W),
\end{equation*}
where $d$ is a certain explicit integer, whose definition is omitted
here as it is not relevant for the discussion in this paper.

It is known that $\Res$ sends $\mathcal Q_W$ to the subcategory
$\mathcal Q_{T,W}$ of $\scr D(\HomE_T\times\HomE_W)$ consisting of
complexes that are isomorphic to finite direct
sums of complexes of the form $(L\boxtimes L')[d]$ for $L\in\mathcal P_T$, 
$L'\in\mathcal P_W$, $d\in\Z$.
Therefore we have an induced $\A$-linear homomorphism
\begin{equation*}
  \Res\colon \scr K(\mathcal Q_V)\to
  \scr K(\mathcal Q_{T,W}) \cong \scr K(\mathcal Q_T)\otimes_{\A}
  \scr K(\mathcal Q_W).
\end{equation*}

We take direct sum over $V$, $T$, $W$ to get a homomorphism of an
algebra with respect to the twisted multiplication \eqref{eq:tm}. It
corresponds to $r$ on $\Uni$ under $\Phi$.

Recall that $\scr K(\mathcal Q_V)$ has an $\A$-basis $(L)$, where
$L$ runs over $\mathcal P_V$. Taking direct sum over $V$, and pulling
back by the isomorphism $\Phi$, we get an $\A$-basis of $\Uni$. This is 
Lusztig's canonical basis. Let us denote it by $\B(\infty)$.

Kashiwara gave an algebraic approach to $\B(\infty)$. See
\cite{Ka-crystal-book} and references therein for detail.
He first introduced an $\A_0$-form $\LL(\infty)$ of $\Un$, where $\A_0
= \{ f\in\Q(q) \mid \text{$f$ is regular at $q=0$}\}$. Then he also
defined a basis, called the {\it crystal base\/} of
$\LL(\infty)/q\LL(\infty)$.
Then he introduced the {\it global crystal base\/} of $\Un$, which
descends to the crystal base of $\LL(\infty)/q\LL(\infty)$.
It turns out that the global crystal base and the canonical base are
the same. See \cite{Grojnowski-Lusztig}.

In this paper, we do not distinguish the crystal base and the
global base, that is the canonical base. We denote both by $\B(\infty)$.

When we want to emphasize that a canonical base element
$b\in\B(\infty)$ is a perverse sheaf, we denote it by $L_b$.

\subsection{Dual canonical base}\label{subsec:dual}

There exists a unique symmetric bilinear form $(\ ,\ )$ on $\Un$
satisfying
\begin{gather*}
  (1, 1) = 1, \qquad (f_i, f_j) = \delta_{ij},
\\
  (r(x),y\otimes z) = (x, yz) \quad\text{for $x$, $y$, $z\in \Un$}.
\end{gather*}
Our normalization is different from \cite[Ch.~1]{Lu-book} and follows
Kashiwara's as in \cite{Kimura,GLSq}.

Under $(\ ,\ )$, we can identify the graded dual algebra of $\Un$, an
algebra with the multiplication given by $r$, with $\Un$ itself.

Let $\B^\up(\infty)$ denote the dual base of $\B(\infty)$ with respect
to $(\ ,\ )$. It is called the {\it dual canonical base\/} of $\Un$.

For $b_1,b_2, b_3\in\B(\infty)$, let us define $r^{b_1,b_2}_{b_3}\in
\A$ by
\begin{equation*}
  r(b_3) = \sum_{b_1,b_2\in \B(\infty)} r^{b_1,b_2}_{b_3} b_1\otimes b_2.
\end{equation*}
Let $b_1^\up$, $b_2^\up$, $b_3^\up\in\B^\up(\infty)$ be the dual
elements corresponding to $b_1$, $b_2$, $b_3$ respectively. Then we
have
\begin{equation*}
  b_1^\up b_2^\up = \sum_{b_3^\up\in\B^\up(\infty)} r^{b_1,b_2}_{b_3} b_3^\up.
\end{equation*}
Thus the structure constant is given by $r^{b_1,b_2}_{b_3}$.

\subsection{Lusztig's lagrangian subvarieties and crystal}\label{subsec:lag}
Let us introduce Lusztig's lagrangian subvariety in the cotangent
space of the space of quiver representations.

The dual space to $\HomE_V$ is
\begin{equation*}
  \HomE_V^* = \bigoplus_{h\in \overline{\Omega}} \Hom(V_{\vout{h}}, V_{\vin{h}}).
\end{equation*}
The group $G_V$ acts on $\HomE_V^*$ in the same way as on $\HomE_V$.

The $G_V$-action preserves the natural pairing between $\HomE_V$ and
$\HomE_V^*$. Considering $\HomE_V\oplus \HomE_V^*$ as a symplectic
manifold, we have the moment map $\mu = (\mu_i) \colon \HomE_V\oplus \HomE_V^*
\to \bigoplus_i\gl(V_i)$ given by 
\begin{equation*}
  \mu_i(B) = \sum_{\vin{h} = i} \varepsilon(h) B_h B_{\overline{h}},
\end{equation*}
where $B$ has components $B_h$ for both $h\in\Omega$ and
$\overline{\Omega}$, and $\varepsilon(h) = 1$ if $h\in\Omega$ and $-1$
otherwise.

Lusztig's lagrangian $\Lambda_V$ \cite{Lu-crystal,Lu-can2} is
defined as
\begin{equation}
  \Lambda_V \defeq \left\{ B\in \HomE_V\oplus \HomE_V^*\,\middle|\,
    \mu(B) = 0, \text{$B$ is nilpotent}\right\}.
\end{equation}

This space parametrizes isomorphism classes of nilpotent
representation of the preprojective algebra associated with the quiver
$(I,\Omega)$, together with a linear base of the underlying vector
space compatible with the $I$-grading. The action of the group $G_V$
is induced by the change of bases.
The preprojective algebra is denoted by $\Lambda$ in this paper.

This is a lagrangian subvariety in $\HomE_V\oplus \HomE_V^*$. (It was
proved that $\Lambda_V$ is half-dimensional in $\HomE_V\oplus
\HomE_V^*$ in \cite[12.3]{Lu-can2}. And the same argument shows that
it is also a lagrangian. Otherwise use \cite[Th.~5.8]{Na-quiver} and
take the limit $W\to\infty$.)

Let $\Irr\Lambda_V$ be the set of irreducible components of
$\Lambda_V$. Lusztig defined a structure of an abstract crystal (see
\cite[\S3]{KS} for the definition) on $\Irr\Lambda_V$ in
\cite{Lu-crystal}, and Kashiwara-Saito proved that it is isomorphic to
the underlying crystal of the canonical base $\B(\infty)$ of
$\Un$ \cite{KS}.

We denote by $\Lambda_b$ the irreducible component of $\Lambda_V$
corresponding to a canonical base element $b\in \B(\infty)$.

\subsection{Dual semicanonical base}\label{subsec:dualsemi}

Let $\Cons(\Lambda_V)$ be the $\Q$-vector space of $\Q$-valued
constructible functions over $\Lambda_V$, which is invariant under the
$G_V$-action. Lusztig defined an operator
$\Cons(\Lambda_T)\times\Cons(\Lambda_W)\to \Cons(\Lambda_V)$ for $V =
T\oplus W$, under which the direct sum $\bigoplus_V \Cons(\Lambda_V)$
is an associative algebra (see \cite[\S12]{Lu-can2}).

If $V = S_i$, then $\Lambda_{S_i}$ is a single point. Let $1_i$ be the
constant function on $\Lambda_{S_i}$ with the value $1$. Let $\Cons_0$
be the subalgebra of $\bigoplus_V \Cons(\Lambda_V)$ generated by the
elements $1_i$ ($i\in I$), and let $\Cons_0(\Lambda_V) = \Cons_0\cap
\Cons(\Lambda_V)$.
Then Lusztig (see \cite[Th.~12.13]{Lu-can2}) proved that $\Cons_0$ is
isomorphic to the universal enveloping algebra $\mathbf U(\mathfrak
n)$ of the lower triangular subalgebra $\mathfrak n$ of $\g$ by $f_i
\mapsto 1_i$.

Note that we have an embedding $\Lambda_T\times \Lambda_W\to
\Lambda_V$ given by the direct sum, where $V = T\oplus W$ as above.
Then the restriction defines an operator $\Cons(\Lambda_V)\to
\Cons(\Lambda_T)\otimes \Cons(\Lambda_W)$. Geiss-Leclerc-Schr\"oer
proved that it sends $\Cons_0(\Lambda_V)$ to
$\Cons_0(\Lambda_T)\otimes\Cons_0(\Lambda_W)$, and gives the natural
cocommutative coproduct on $\mathbf U(\mathfrak n)$ under the
isomorphism $\Cons_0\cong\mathbf U(\mathfrak n)$ (see \cite[\S4]{GLS1}).

Let $Y$ be an irreducible component of $\Lambda_V$. Then consider the
functional $\rho_Y\colon \Cons(\Lambda_V)\to \Q$ given by taking the
value on a dense open subset of $Y$. Then $\{ \rho_Y \mid
Y\in\Irr\Lambda_V\}$ gives a base of $\Cons_0(\Lambda_V)$. This
follows from \cite[\S3]{Lu-affine} together with the result of
Kashiwara-Saito mentioned above. Under the isomorphism $\mathbf
U(\mathfrak n)\cong \Cons_0$, the base $\rho_Y$ is called the {\it
  dual semicanonical base\/} of $\mathbf U(\mathfrak
n)^*_{\operatorname{gr}}$, where $\mathbf U(\mathfrak
n)^*_{\operatorname{gr}}$ denote the graded dual of $\mathbf
U(\mathfrak n)$.

We have a natural bijection $b^\up\mapsto \rho_{\Lambda_b}$ between
the dual canonical base and the dual semicanonical base.
However $\rho_{\Lambda_b}$ is different from the specialization of
$b^\up$ at $q=1$ in general. See \cite[\S1.5]{GLS1} for a counter-example.

\section{Cluster algebras and quantum unipotent subgroups}

We fix a Weyl group element $w$ throughout this section. Let
$\Delta_w^+ = \Delta^+\cap w(-\Delta^+)$, where $\Delta^+$ is the set
of positive roots.
Then $\mathfrak n(w) = \bigoplus_{\alpha\in\Delta_w^+} \g_{-\alpha}$
is a Lie subalgebra of $\g$, where $\g_{-\alpha}$ is the root subspace
corresponding to the root $-\alpha$.

\subsection{Quantum unipotent subgroup}

Let us briefly recall the $q$-analog of the universal enveloping
algebra $\mathbf U(\mathfrak n(w))$ of $\mathfrak n(w)$, denoted by
$\Un(w)$. See \cite[Ch.~40]{Lu-book}, \cite[\S4]{Kimura} and
\cite{GLSq} for more detail. (It is denoted by $A_q(\mathfrak n(w))$
in \cite{GLSq}.)

Let $T_i$ be the braid group operator corresponding to $i\in I$, where
$T_i = T_{i,1}''$ in the notation in \cite{Lu-book}.
Choose a reduced expression $w = s_{i_1}s_{i_2}\cdots
s_{i_{\ell}}$. Then it gives $\beta_p = s_{i_1} s_{i_2}\cdots
s_{i_{p-1}}(\alpha_{i_p})$, and we have $\Delta_w^+ = \{ \beta_p\}_{1\le
  p\le \ell}$. We define a {\it root vector}
\begin{equation*}
  T_{i_1} T_{i_2}\cdots T_{i_{p-1}}(f_{i_p}).
\end{equation*}
Let $\bc = (c_1,\dots,c_\ell)\in\Z_{\ge 0}^\ell$. We multiply
$q$-divided powers of root vectors in the order given by $\beta_1$,
\dots, $\beta_\ell$:
\begin{equation*}
  L(\bc) = f_{i_1}^{(c_1)} T_{i_1}(f_{i_2}^{(c_2)})\cdots
  (T_{i_1}\cdots T_{i_{\ell-1}})(f_{i_\ell}^{(c_\ell)}).
\end{equation*}
Then the $\Q(q)$-subspace spanned by $L(\bc)$ ($\bc\in\Z_{\ge
  0}^\ell$) is independent of the choice of a reduced expression of $w$.
This subspace is $\Un(w)$. Moreover $L(\bc)$ gives a basis of
$\Un(w)$. It can be shown that $\Un(w)$ is a subalgebra of $\Un$.

It is known that $\Un(w)$ is compatible with the dual canonical base,
i.e., $\Un(w)\cap \B^\up(\infty)$ is a base of $\Un(w)$. This is an
interpretation of the main result due to Lusztig
\cite[Th.~1.2]{Lu-braid}, based on an earlier work by Saito
\cite{Saito-PBW}. (See \cite[Th.~4.25]{Kimura} for the current
statement.)

Let $\B^\up(w) \defeq \Un(w)\cap \B^\up(\infty)$.
Let $\B(w)\subset \B(\infty)$ be the corresponding subset in the
canonical base.
Then \cite[Prop.~8.3]{Lu-braid} gives a parametrization of $\B(w)$ as
follows.
Let $\LL(\infty)$ be the $\A_0$-form used in the definition of the
crystal base. Then one shows $L(\bc)\in\LL(\infty)$ and the set $\{
L(\bc)\bmod q\LL(\infty)\}$ is equal to $\B(w)$, where $\B(w)$ is
considered as a subset of $\LL(\infty)/q\LL(\infty)$. Let us denote by
$b(\bc)\in\B(w)$ the canonical base element corresponding to
$L(\bc)\bmod q\LL(\infty)$. Therefore $b(\bc)\equiv L(\bc) \mod
q\LL(\infty)$.

The argument in the proof of \cite[Th.~3.13]{MR2066942} shows that the
transition matrix between the base $\{ L(\bc)\}$ and $\{ b(\bc)\}$ is
upper triangular with respect to the lexicographic order on $\{ \bc\}$.

It is known that $\{ L(\bc)\}$ is orthogonal with respect to $(\ ,\ )$
(see \cite[Prop.~40.2.4]{Lu-book}). Therefore we can deduce the
corresponding relation between $b^\up(\bc)$ and $L^\up(\bc) =
L(\bc)/(L(\bc),L(\bc))$, where $b^\up(\bc)$ is the dual canonical base
element corresponding to $b(\bc)$.  (See \cite[Th.~4.29]{Kimura}.)

\subsection{$\B(w)$ and Kashiwara operators}\label{subsec:Bw}

Let us give a characterization of $\B(w)$ in terms of Kashiwara
operators on $\B(\infty)$. Recall that $\B(\infty)$ is an abstract
crystal, and hence has maps $\wt\colon \B(\infty)\to P$,
$\varepsilon_i\colon \B(\infty) \to \Z$, $\varphi_i\colon
\B(\infty)\to \Z$ ($i\in I$) together with Kashiwara operators
$\te_i\colon\B(\infty)\to \B(\infty)\sqcup\{0\}$,
$\tf_i\colon\B(\infty)\to \B(\infty)\sqcup\{0\}$ satisfying certain
axioms.
We denote by $u_\infty$ the element in $\B(\infty)$ corresponding to
$1$ in $\Un$.
There is also an operator $\ast\colon \B(\infty)\to \B(\infty)$, which
corresponds to the anti-involution $\ast\colon\Un\to \Un$ given by
$f_i\mapsto f_i$.
Therefore we have another set of maps and operators
$\varepsilon_i^* = \varepsilon_i\ast$,
$\varphi_i^* = \varphi_i\ast$,
$\te_i^* = \ast\te_i\ast$, $\tf_i^* = \ast\tf_i\ast$.

Let $w =
s_{i_1}s_{i_2}\cdots s_{i_\ell}$ as above. For $i=i_1$, we have
\begin{gather*}
  \varepsilon_i(b) = c_1,
\\
  \te_i b(\bc) =
  \begin{cases}
    b(\bc') & \text{if $c_1\neq 0$}, \\
    0 & \text{if $c_1 = 0$},
  \end{cases}\qquad
  \tf_i b(\bc) = b(\bc''),
\end{gather*}
where $\bc' = (c_1-1,c_2,\dots, c_\ell)$, $\bc'' = (c_1+1,c_2,\dots,
c_\ell)$.
In particular, $\B(w)$ is invariant under $\te_i$, $\tf_i$.

Saito \cite[Cor.~3.4.8]{Saito-PBW} introduced a bijection 
\begin{equation*}
  \Lambda_i\colon \{ b\in\B(\infty) \mid \varepsilon_i^*(b) = 0\}
  \to \{ b\in\B(\infty) \mid \varepsilon_i(b) = 0\}
\end{equation*}
by
\begin{equation*}
  \Lambda_i(b) = (\tf_i^*)^{\varphi_i(b)} (\te_i)^{\varepsilon_i(b)} b,
\qquad
  \Lambda_i^{-1}(b) = (\tf_i)^{\varphi_i^*(b)} (\te_i^*)^{\varepsilon_i^*(b)} b.
\end{equation*}
This is related to the braid group operator as follows. Let us
consider $w = s_{i_1} s_{i_2}\cdots s_{i_\ell}$ and $w' = s_{i_2}
\cdots s_{i_\ell} s_{i_1} = s_{i_1} w s_{i_1}$ and corresponding PBW
base elements
\begin{gather*}
  L = f_{i_1}^{(c_1)} T_{i_1}(f_{i_2}^{(c_2)})\cdots
  (T_{i_1}\cdots T_{i_{\ell-1}})(f_{i_\ell}^{(c_\ell)}),
\\
  L' = f_{i_2}^{(c_2')} T_{i_2}(f_{i_3}^{(c_3')}) \cdots
  (T_{i_2}\cdots T_{i_\ell})(f_{i_1}^{(c_1')}),
\end{gather*}
for $(c_1,\dots,c_\ell)\in\Z_{\ge 0}^\ell$, $(c_2',\dots,
c_\ell',c_1')\in\Z_{\ge 0}^\ell$. If $c_1 = 0$, $c_2 = c'_2$, \dots,
$c_\ell = c'_\ell$, $0 = c'_1$, we have $L = T_{i_1} L'$. Let $b$,
$b'$ be the canonical base elements corresponding to $L$ and $L'$
respectively.
Then Saito \cite[Prop.~3.4.7]{Saito-PBW} proved that the corresponding
canonical base elements are related by $b = \Lambda_{i_1} b'$.
As a corollary of this result, we have a bijection
\begin{equation*}
  \{ b\in \B(w') \mid \varepsilon_i^*(b) = 0\}
  \xrightarrow{\Lambda_i}
  \{ b\in \B(w) \mid \varepsilon_i(b) = 0\}.
\end{equation*}
This together with the invariance of $\B(w)$ under $\te_{i_1}$,
$\tf_{i_1}$ gives a characterization of $\B(w)$ inductively in the
length of $w$, starting from $\B(1) = \{ u_\infty\}$.

\subsection{A subcategory $\mathcal C_w$}\label{subsec:Cw}

In view of \subsecref{subsec:lag} it is natural to look for a
characterization of $\B(w)$ in terms of Lusztig's lagrangian
subvarieties $\Lambda_V$, or the representation theory of the
preprojective algebra.
It turns out to be related to the subcategory introduced by
Buan-Iyama-Reiten-Scott \cite{BIRS}, and further studied by
Geiss-Leclerc-Schr\"oer \cite{GLSc} and Baumann-{Kamnitzer}-{Tingley}
\cite{BKT}.

We do not recall the definition of the subcategory of the category of
finite-dimensional nilpotent representations of the preprojective
algebra, denoted by $\mathcal C_w$ following \cite{GLSc}, here. This
is because there are many equivalent definitions, and the author does
not know what is the best for our purpose. See the above papers.

Let $\Lambda_V^w = \{ B\in\Lambda_V \mid B\in\mathcal C_w\}$, where we
identify $B$ with the corresponding representation of the
preprojective algebra.
This is an open subvariety in $\Lambda_V$ (see \cite[Lem.~7.2]{GLSc2}).
We set 
\begin{equation*}
  C_0^w(\Lambda_V) =
  \{ f\in C_0(\Lambda_V)\mid \text{$f(X) = 0$ for $X\in \mathcal C_w$}\}.
\end{equation*}
Therefore $C_0(\Lambda_V)/C_0^w(\Lambda_V)$ consists of constructible
functions on $\Lambda_V^w$.

Let us consider
\begin{equation*}
  C_0^w(\Lambda_V)^\perp = 
  \{ \xi\in C_0(\Lambda_V)^* \mid 
  \text{$\langle f, \xi\rangle = 0$ for $f\in C_0^w(\Lambda_V)$}\}.
\end{equation*}
This space is spanned by an evaluation at a point $X\in \Lambda_V^w$.
As $\mathcal C_w$ is an additive category (in fact, it is closed under
extensions), we have
\begin{equation*}
  C_0^w(\Lambda_T)^\perp \cdot C_0^w(\Lambda_W)^\perp
  \subset C_0^w(\Lambda_V)^\perp,
\end{equation*}
where the multiplication is given by the transpose of $r$, the natural
cocommutative coproduct on $\mathbf U(\mathfrak n)$. This follows from
the interpretation of $r$ explained in \subsecref{subsec:dualsemi}.

Therefore $\bigoplus C_0^w(\Lambda_V)^\perp \subset(C_0)_{\gr}^*\cong
\mathbf U(\mathfrak n)^*_{\gr}$ is a subalgebra.
By \cite[Th.~3.3]{GLSc} it is the $\C[N(w)]$, which is the $q=1$ limit
of $\Uq^-(w)$ (\cite[Th.~4.44]{Kimura}).

By \cite[Th.~3.2]{GLSc} $\C[N(w)]$ is compatible with the dual
semicanonical basis, i.e., the intersection $\{ \rho_{\Lambda_b} \mid
b\in\B(\infty) \} \cap\C[N(w)]$ is a base of $\C[N(w)]$.
This base consists of $\rho_{\Lambda_b}$ such that $\Lambda_b$
intersects with the open subvariety $\Lambda_V^w$. The intersection
$\Lambda_b\cap \Lambda_V^w$ is an open dense subset of $\Lambda_b$.

Finally we have
\begin{equation*}
  \{ \rho_{\Lambda_b} \mid
  b\in\B(\infty) \} \cap\C[N(w)]
  = \{ \rho_{\Lambda_b} \mid b\in \B(w) \}.
\end{equation*}
This follows from \cite[\S5.5]{BKT} and the characterization of
$\B(w)$ in \subsecref{subsec:Bw}.

\subsection{Cluster algebras and $\mathcal C_w$}\label{subsec:cluster}

Geiss-Leclerc-Schr\"oer \cite{GLSc} have introduced a structure of the
cluster algebra, in the sense of Fomin-Zelevinsky \cite{FomZel}, on
$\C[N(w)]$. One of their main results says that dual semicanonical
base of $\C[N(w)]$ contains {\it cluster monomials}.
We review their theory only briefly here. See the original paper for
more detail.
The construction is based on $\mathcal C_w$ in \subsecref{subsec:Cw}.

A $\Lambda$-module $T$ is {\it rigid\/} if $\Ext^1_\Lambda(T,T) =
0$. It is easy to see from the formula
\(
  \dim\Ext^1_\Lambda(T,T) = 2\dim\Hom_\Lambda(T,T) - (\dim V,\dim V)
\)
(see e.g., \cite[Lem.~2.1]{GLSc})
that this is equivalent to say that the orbit through $T$ is open
in $\Lambda_V$, where $V$ is the $I$-graded vector space underlying $T$,
and $(\ ,\ )$ is the Cartan matrix for the graph $\mathcal G$.
We say $T$ is {\it $\mathcal C_w$-maximal rigid\/} if
$\Ext^1_\Lambda(T\oplus X,X) = 0$ with $X\in\mathcal C_w$ implies $X$
is in $\operatorname{add}(T)$, the subcategory of modules which are
isomorphic to finite direct sums of direct summands of $T$.
In $\mathcal C_w$, there is a distinguished $\mathcal C_w$-maximal
module, denoted by $V_{\mathbf i}$ in \cite{GLSc}, where $\mathbf i =
(i_1,\dots,i_\ell)$ is a reduced expression of $w$.
It is conjectured that every $\mathcal C_w$-maximal rigid module $T$
is {\it reachable\/}, i.e., it is obtained from $V_{\mathbf i}$ using
a sequence of operations, called {\it mutations}. This will be
recalled below.

Let $T$ be a reachable $\mathcal C_w$-maximal rigid module, and $T =
T_1\oplus\cdots \oplus T_\ell$ be the decomposition into
indecomposables.
We assume that $T$ is {\it basic\/}, which means $T_i$ are pairwise
non-isomorphic.
In this case the number of summands $\ell$ is known to be equal to the
length of $w$.
If $R\in \operatorname{add}(T)$, it is a rigid $\Lambda$-module, and
hence the closure of the corresponding orbit is an irreducible
component of $\Lambda_V$ for an appropriate choice of $V$.
We denote it by $\rho_R\in \C[N(w)]$ the corresponding dual
semicanonical base elements.

From the identification of the coproduct $r$ on $\mathbf U(\mathfrak
n)^*_{\gr}$, we see that 
\begin{equation}\label{eq:monomial}
  \rho_R = \rho_{T_1}^{c_1}\cdots \rho_{T_\ell}^{c_\ell},
\end{equation}
where $R = T_1^{\oplus c_1}\oplus \cdots \oplus T_\ell^{c_\ell}$
with $c_i\in\Z_{\ge 0}$.
In the context of the cluster algebra theory, $\rho_{T_i}$ ($1\le i\le
\ell$) are called {\it cluster variables\/} and $\rho_R$ is a {\it
  cluster monomial}.

Let $T$ be a basic $\mathcal C_w$-maximal rigid module, and $T_k$ be a
non-projective indecomposable direct summand of $T$. The mutation
$\mu_k(T) $ is a new basic $\mathcal C_w$-maximal rigid module of the
form $(T/T_k)\oplus T_k^*$, where $T_k^*$ is another indecomposable
module.  Such a $\mu_k(T)$ exists and is uniquely determined from 
$T$ and $T_k$ (see \cite[Prop.~2.19]{GLSc} and the reference therein).
We have $\dim \Ext^1_\Lambda(T_k,T_k^*) =
\dim\Ext^1_\Lambda(T_k^*,T_k) = 1$ and
\begin{equation}\label{eq:exchange}
  \rho_{T_k}\rho_{T_k^*} = \rho_{T'} + \rho_{T''},
\end{equation}
where $T'$, $T''\in\operatorname{add}(T/T_k)$ are given by short exact
sequences
\begin{equation*}
  0 \to T_k \to T'\to T_k^*\to 0, \qquad
  0 \to T_k^* \to T''\to T_k\to 0
\end{equation*}
respectively.

Geiss-Leclerc-Schr\"oer \cite{GLSq} have obtained a $q$-analog of the
result explained above, namely they have introduced a structure of the
quantum cluster algebra, in the sense of Berenstein-Zelevinsky
\cite{BerZel}, on the quantum unipotent subgroup $\Un(w)$.
(This result was conjectured in \cite{Kimura}.)
The construction is again based on $\mathcal C_w$.
Let $T$ be a reachable $\mathcal C_w$-maximal rigid module. For
$R\in\operatorname{add}(T)$, there is an element $Y_R\in\Un(w)$, which
satisfies the $q$-analog of (\ref{eq:monomial}, \ref{eq:exchange}):
\begin{equation}\label{eq:two_formulas}
  \begin{gathered}
    Y_R = q^{-\alpha_R} Y_{T_1}^{c_1}\cdots Y_{T_\ell}^{c_\ell},
\\
    Y_{T_k^*} Y_{T_k} = q^{[T_k^*,T_k]} (q^{-1} Y_{T'} + Y_{T''})
  \end{gathered}
\end{equation}
for appropriate explicit $\alpha_R$, $[T_k^*, T_k]\in\Z$.
See \cite[(10.17) and Prop.~10.5]{GLSq}.

One of the main conjectures in this theory is
\begin{Conjecture}\label{conj:qmon}
  All quantum cluster monomials $Y_R$ are contained in $\B^\up(w)$.
\end{Conjecture}

This is closely related to their earlier open orbit conjecture
\cite[\S18.3]{GLSc}:
\begin{Conjecture}
  Suppose that an irreducible component $\Lambda_b$ of $\Lambda_V$
  contains an open $G_V$-orbit, then the corresponding dual
  semicanonical base element $\rho_{\Lambda_b}$ is equal to the
  specialization of the corresponding canonical base element $b$ at
  $q=1$.
\end{Conjecture}

In fact, this is implied by Conj.~\ref{conj:qmon} for reachable rigid
modules.

\section{Singular supports under the restriction}

Let $SS(L)$ denote the singular support of a complex $L\in\mathcal Q_V$.
See \cite{KaSha} for the definition. Lusztig proved that
$SS(L)\subset\Lambda_V$ \cite[13.6]{Lu-can2}.
In fact, we have finer estimates
\begin{equation}\label{eq:SSest}
  \Lambda_b \subset SS(L_b) \subset
  \bigcup_{\substack{b'\in \B(\infty)\\ 
      \forall i \ \varepsilon_i(b')\ge \varepsilon_i(b)}}
  \Lambda_{b'}.
\end{equation}
See \cite[Thm.~6.2.2]{KS}, but note that there is a misprint. See
\cite[Lem.~8.2.1]{KS} for the correct statement.

These estimates give us some relation between the canonical base and
$\bigsqcup_V \Irr\Lambda_V$ via singular supports. We study
the behavior of singular supports under the functor $\Res$ in this
section.

\subsection{The statement}

In order to state the result, we prepare notation.

Let $f\colon Y\to X$ be a morphism. Let $TX$, $TY$ (resp.\ $T^*X$,
$T^*Y$) be tangent (resp.\ cotangent) bundles of $X$, $Y$ respectively.
Let $f^{-1}TX = Y\times_X TX$ (resp.\ $f^{-1} T^*X = Y\times_X T^* X$)
be the pull-back of $TX$ (resp.\ $T^*X$) by $f$. We have associated
morphisms
\begin{equation*}
  T^* Y \xleftarrow{\lsp{t}f'} f^{-1} T^*X \xrightarrow{f_\pi} T^*X,
\end{equation*}
where $\lsp{t}f'$ is the transpose of the differential
$f' \colon TY \to f^{-1} TX = Y\times_X TX$.

We apply this construction for the morphisms $\iota$, $\kappa$ to
get morphisms
\begin{equation*}
  \begin{gathered}
  T^* E(W) \xleftarrow{\lsp{t}\iota'} 
  \iota^{-1} T^* \HomE_V
  \xrightarrow{\iota_\pi} T^* \HomE_V ,
\\
  T^* E(W) \xleftarrow{\lsp{t}\kappa'} 
  \kappa^{-1} T^*(\HomE_T\times \HomE_W)
  \xrightarrow{\kappa_\pi} T^* (\HomE_T\times \HomE_W).
  \end{gathered}
\end{equation*}

\begin{Theorem}\label{thm:est}
  We have
  \begin{equation*}
    SS(\Res(L))\subset \kappa_\pi(\lsp{t}\kappa^{\prime-1}(
    \lsp{t}\iota^{\prime} (\iota_\pi^{-1} (SS(L))))).
  \end{equation*}
\end{Theorem}

Let us remark that the proof shows the following statement. Choose a
complementary subspace of $W$ in $V$, and identify $V$ with $W\oplus
T$. Then we have the induced embedding
$T^*(\HomE_T\times\HomE_W)\subset T^*\HomE_V$. Then we have
\begin{equation}\label{eq:0sect}
  SS(\Res(L))\subset T^*(\HomE_T\times\HomE_W)\cap SS(L).
\end{equation}

The proof occupies the rest of this section.

%\subsection{A non-characteristic property}
\subsection{Inverse image}

If $\iota$ were smooth, we would have $SS(\iota^* L)\subset
\lsp{t}\iota^{\prime} (\iota_\pi^{-1} (SS(L)))$ by
\cite[Prop.~5.4.5]{KaSha}. And if $\kappa$ were proper, we would have
$SS(\kappa_!\iota^* L) \subset\break[3]
\kappa_\pi(\lsp{t}\kappa^{\prime-1}(SS(\iota^*L)))$ by
\cite[Prop.~5.4.4]{KaSha}. Therefore the assertion follows. However
neither are true, so we need more refined versions of these estimates.

In order to study the behavior of the singular support under the
pull-back by a non-smooth morphism, we need several more notions
related to cotangent manifolds from \cite[Ch.~VI]{KaSha}.

\begin{NB}
  In order to study the behavior of the singular support under the
  pull-back, we recall the notion of a {\it non-characteristic\/}
  property for a morphism with respect to a closed conic subset of a
  subset of the cotangent bundle. See \cite[Ch.~VI]{KaSha} for more
  detail.
\end{NB}

We first recall the normal cone to $S$ along $M$ briefly.
Suppose that $M$ is a closed submanifold of a manifold $X$. Let
$T_MX$ denote the normal bundle of $M$ in $X$.
Then one can define a new manifold $\tilde{X}_M$, which connects $X$
and $T_MX$ in the following way: there are two maps
\begin{equation*}
  p\colon \tilde{X}_M \to X, \quad t\colon \tilde{X}_M\to \R
\end{equation*}
such that $p^{-1}(X\setminus M)$, $t^{-1}(\R\setminus\{0\})$ and
$t^{-1}(0)$ are isomorphic to $(X\setminus M)\times(\R\setminus\{0\})$,
$X\times(\R\setminus\{0\})$ and $T_MX$ respectively.
In our application, $M$ is the zero section of a vector bundle $X$,
and hence the normal bundle $T_MX$ is $X$ itself. In this case,
$\tilde{X}_M$ is $X\times \R$ and $p$, $t$ are the first and second
projections. A general definition is in \cite[\S4.1]{KaSha} for
an interested reader.

Let $\Omega$ be the inverse image of $\R^+$ under $t$, and $\tilde p$
the restriction of $p$ to $\Omega$.

Let $S$ be a subset of $X$. The {\it normal cone to $S$ along $M$},
denoted by $C_M(S)$ is defined by
\begin{equation*}
  C_M(S) \defeq T_M X \cap \overline{\tilde p^{-1}(S)}.
\end{equation*}
If $M$ is the zero section of a vector bundle $X$ as before, we have
$C_M(S)$ is identified with $S$ itself under $T_MX \cong X$.

Now we return back to a morphism $f\colon Y\to X$. We assume that $f$
is a closed embedding for simplicity. We consider the conormal bundle
$T^*_YX$. We denote the projection $T^*_YX\to Y$ by $p$. We have
morphisms
\(
  T^* ({T}^*_Y X) \xleftarrow{\lsp{t}{{p}}'} {p}^{-1} T^*Y 
  \xrightarrow{{p}_\pi} T^*Y
\)
as before.

\begin{NB}
Now we return back to a morphism $f\colon Y\to X$. We assume that $f$
is a closed embedding for simplicity. We consider the conormal bundle
$T^*_YX$ and the complement of the zero section $\dot{T}^*_YX$. We
denote the projection $T^*_YX\to Y$ by $p$ and its restriction to
$\dot{T}^*_YX$ by $\dot{p}$. We have morphisms
\begin{equation*}
  T^* (\dot{T}^*_Y X) \xleftarrow{\lsp{t}{\dot{p}}'} \dot{p}^{-1} T^*Y 
  \xrightarrow{\dot{p}_\pi} T^*Y
\end{equation*}
as before.
\end{NB}

We consider $T^*Y$ as a submanifold of $T^*(T^*_YX)$ via the
composition 
\begin{equation*}
  T^*Y\hookrightarrow p^{-1}T^* Y \xrightarrow{\lsp{t}p'}
  T^*(T^*_Y X).
\end{equation*}
See \cite[(5.5.10)]{KaSha}.

We consider ${T}^*_YX$ as a closed submanifold of ${T}^*X$. So we can
define the normal cone to a subset of $T^*X$ along $T^*_YX$.

\begin{NB}
We consider $\dot{T}^*_YX$ as a closed submanifold of $\dot{T}^*X$,
the complement of the zero section of the cotangent bundle $T^*X$ of $X$.
\end{NB}

If $f\colon Y\to X$ is the embedding of the zero section to a vector
bundle $X$, we can identify $T^*_Y X$ with the dual vector bundle
$X^*$. Then $T^*_Y X \to T^*X$ is also the embedding of the zero
section to a vector bundle. In fact, we have a natural identification
$T^*X \cong T^* X^*$, therefore $T^*_Y X \to T^*X$ is identified with
$X^* \to T^*X^*$. Therefore $C_{T^*_YX}(A)$ is identified with $A$
itself under the isomorphism $T_{T^*_YX}(T^*X)\cong T^*X$.

Note also that $T^*(T^*_YX)$ is identified with $T^*X^*\cong
T^*X$. Under this identification, we have an isomorphism $p^{-1}T^*Y
\cong f^{-1}T^*X$ which gives a commutative diagram
\begin{equation}\label{eq:CD}
  \begin{CD}
    T^*X @<{f_{\pi}}<< f^{-1} T^*X @>{\lsp{t}f'}>> T^* Y
\\
   @V{\cong}VV @V{\cong}VV @|
\\
    T^*(T^*_Y X) @<<{\lsp{t}p'}< p^{-1} T^* Y @>>{p_\pi}> T^* Y.
  \end{CD}
\end{equation}

\begin{NB}
  Let us take a coordinate system $(y,x)$ for $X$, where $y$ is a
  coordinate system for $Y$ and $x$ is a fiber coordinate of the vector
  bundle $X\to Y$. Take the corresponding coordinate system
  $(y,x;\xi,\zeta)$ for $T^*X$. Then we have
  \begin{equation*}
    \begin{CD}
      T^*(T^*_Y X) @<<{\lsp{t}p'}< p^{-1} T^* Y @>>{p_\pi}> T^* Y
\\
      \rotatebox{90}{$\in$} @. \rotatebox{90}{$\in$} @. \rotatebox{90}{$\in$}
\\
     (y,x;\xi,\zeta) @. (y,0;\xi,\zeta) @. (y,0;\xi,0)
    \end{CD}
  \end{equation*}
\end{NB}

Let $A$ be a conic subset of $T^*X$, i.e., invariant under the
$\R^+$-action, the multiplication on fibers. We define
\begin{equation*}
  f^\#(A) \defeq T^*Y\cap C_{{T}^*_YX}(A).
\end{equation*}
Here we identify $T_{{T}^*_YX}{T}^*X$ with $T^*({T}^*_YX)$.
See \cite[(6.2.3)]{KaSha}. Moreover we also have
\begin{equation*}
  f^\#(A)  = p_\pi\lsp{t}p^{\prime-1}(C_{{T}^*_YX}(A)).
\end{equation*}
See \cite[Lem.~6.2.1]{KaSha}.

If $f\colon Y\to X$ is the embedding of the zero section to a vector
bundle $X$, we have
\begin{equation}\label{eq:fs}
  f^\#(A) = T^* Y \cap A = \lsp{t}f' f_\pi^{-1}(A)
\end{equation}
by the commutative diagram \eqref{eq:CD}.

\begin{NB}
Let $A$ be a conic subset of $T^*X$, i.e., invariant under the
$\R^+$-action, the multiplication on fibers. We define
\begin{equation*}
  f_\infty^\#(A) \defeq
  \dot{p}_\pi\lsp{t}{\dot{p}}^{\prime-1}(C_{\dot{T}^*_YX}(A)).
\end{equation*}
Here we identify $T_{\dot{T}^*_YX}\dot{T}^*X$ with $T^*(\dot{T}^*_YX)$.
See \cite[(6.2.3)]{KaSha}.
\end{NB}

\begin{NB}
Let $V$ be a subset of $T^*Y$. We say $f$ is {\it
  non-characteristic\/} for $A$ on $V$ if $f_\infty^\#(A)\cap
V=\emptyset$ \cite[Def.~6.2.7]{KaSha}.
Let $F\in\scr D(X)$. We say $f$ is {\it non-characteristic\/} for $F$
on $V$ if $f$ is so for $SS(F)$ on $V$.
In this case we have \cite[Cor.~6.4.4]{KaSha}
\begin{equation}
  SS(f^{-1}F) \cap V \subset \lsp{t}f'(f_\pi^{-1}(SS(F))).
\end{equation}
\end{NB}

\subsection{Spaces}

Let us describe relevant spaces explicitly.

As $E(W)$ is a linear subspace of $\HomE_V$, we have $T^* E(W) =
E(W)\times E(W)^*$ and $E(W)^*$ is identified with
$\HomE_V^* / E(W)^\perp$, where
\begin{equation*}
  E(W)^\perp = 
  \left\{ B'\in \HomE_V^*\, \middle|\,
    B'(W)=0, \Ima B'\subset W \right\}.
\end{equation*}
Taking a complementary subspace of $W$ in $V$, we identify $T$ as an
$I$-graded {\it subspace\/} of $V$. We then have the direct sum
decomposition $V \cong W\oplus T$ and the induced projection $V\to W$.
Then we have matrix notations of $B$ and $B'$:
\begin{equation*}
  B =
  \begin{pmatrix}
    B_{TT} & 0 
    \\
    B_{WT} & B_{WW}
  \end{pmatrix},
\quad
  B' =
  \begin{pmatrix}
    B'_{TT} & B'_{TW}
    \\
    0 & B'_{WW}
  \end{pmatrix}.
\end{equation*}
Similarly the space $\iota^{-1}(T^*\HomE_V)$ is nothing but
$E(W)\times \HomE_V^*$, and identified with the space of linear maps
$B$, $B'$ of the forms
\begin{equation*}
  B =
  \begin{pmatrix}
    B_{TT} & 0 
    \\
    B_{WT} & B_{WW}
  \end{pmatrix},
\quad
  B' =
  \begin{pmatrix}
    B'_{TT} & B'_{TW}
    \\
    B'_{WT} & B'_{WW}
  \end{pmatrix}.
\end{equation*}

The morphism $\lsp{t}\iota'\colon \iota^{-1}(T^*\HomE_V)\to T^*E(W)$
is induced by the projection $\HomE_V^*\to E(W)^*$. In the matrix
notation, it is given by forgetting the component $B'_{WT}$.

The morphism $\iota_\pi\colon \iota^{-1}(T^*\HomE_V)\to T^*\HomE_V$ is
the embedding $E(W)\times \HomE_V^*\to \HomE_V\times \HomE_V^*$.

For $(B,B')\in T^*(\HomE_T\times\HomE_W)$, we have
\begin{equation*}
  B =
  \begin{pmatrix}
    B_{TT} & 0 
    \\
    0 & B_{WW}
  \end{pmatrix},
\quad
  B' =
  \begin{pmatrix}
    B'_{TT} & 0
    \\
    0 & B'_{WW}
  \end{pmatrix}.
\end{equation*}

For $(B,B')\in \kappa^{-1}(T^*(\HomE_T\times\HomE_W))$, we have
\begin{equation*}\label{eq:kappa}
  B =
  \begin{pmatrix}
    B_{TT} & 0 
    \\
    B_{WT} & B_{WW}
  \end{pmatrix},
\quad
  B' =
  \begin{pmatrix}
    B'_{TT} & 0
    \\
    0 & B'_{WW}
  \end{pmatrix}.
\end{equation*}

The morphisms
\begin{equation*}
  T^*E(W) \xleftarrow{\lsp{t}\kappa'} \kappa^{-1}(T^*(\HomE_T\times\HomE_W))
  \xrightarrow{\kappa_\pi} T^*(\HomE_T\times\HomE_W)
\end{equation*}
are given by taking appropriate matrix entries of $B$, $B'$.

Since we have chosen an isomorphism $V\cong W\oplus T$, we have the
projection $p\colon \HomE_V\to E(W)$ which gives a structure of a
vector bundle so that $\iota$ is the embedding of the zero section.
Therefore we have the commutative diagram \eqref{eq:CD} for $f =
\iota$, and hence
\begin{equation}\label{eq:is}
  \iota^\#(A) = T^*E(W) \cap A =
  \lsp{t}\iota' \iota_{\pi}^{-1}(A)
\end{equation}
by \eqref{eq:fs}.

\begin{NB}
The conormal bundle $T^*_{E(W)}\HomE_V$ of $E(W)$ in $\HomE_V$ is
$E(W)\times E(W)^\perp$.  In the matrix notation, it is identified with
the space of linear maps $B$, $B'$ of the forms
\begin{equation*}
    B =
  \begin{pmatrix}
    B_{TT} & 0 
    \\
    B_{WT} & B_{WW}
  \end{pmatrix},
\quad
  B' =
  \begin{pmatrix}
    0 & 0
    \\
    B'_{WT} & 0
  \end{pmatrix}.
\end{equation*}
\begin{NB2}
The complement $\dot{T}^*_{E(W)}\HomE_V$ of the zero section is
$E(W)\times \dot{E}(W)^\perp$, where $\dot{E}(W)^\perp$ is the
complement of $0$. It is given by the condition $B'_{WT}\neq 0$.
\end{NB2}

The projection $p\colon T^*_{E(W)}\HomE_V\to E(W)$ is given by
the first projection $E(W)\times E(W)^\perp\to E(W)$.
We have
\begin{equation*}
  \begin{split}
  T^* (T^*_{E(W)} \HomE_V)
  & \cong E(W)\times E(W)^* \times E(W)^\perp \times (E(W)^\perp)^*
\\
  & \cong \HomE_V \times \HomE_V^* \cong T^* \HomE_V,
  \end{split}
\end{equation*}
where we have used isomorphisms
$E(W)^*\times E(W)^\perp \cong \HomE_V^*$ and
$E(W)\times (E(W)^\perp)^* \cong \HomE_V$.
The embedding $T^* E(W)\to T^* (T^*_{E(W)} \HomE_V)$ is identified with
the inclusion of
$E(W)\times E(W)^* \cong E(W)\times E(W)^*\times\{0\}\times\{0\}$.

We also have
\begin{equation*}
  p^{-1}T^* E(W)
  \cong E(W)\times E(W)^* \times E(W)^\perp.
\end{equation*}
For $(B,B')\in p^{-1} T^* E(W)$, we have
\begin{equation*}
    B =
  \begin{pmatrix}
    B_{TT} & 0 
    \\
    B_{WT} & B_{WW}
  \end{pmatrix},
\quad
  B' =
  \begin{pmatrix}
    B'_{TT} & B'_{TW}
    \\
    B'_{WT} & B'_{WW}
  \end{pmatrix}.
\end{equation*}
The morphisms
\(
      T^* (T^*_Y X) \xleftarrow{\lsp{t}p'} p^{-1} T^*Y 
      \xrightarrow{p_\pi} T^*Y
\)
are identified with
\begin{multline*}
  E(W)\times E(W)^*\times E(W)^\perp \times (E(W)^\perp)^*
\\
  \leftarrow
  E(W)\times E(W)^* \times E(W)^\perp
  \rightarrow
  E(W)\times E(W)^*,
\end{multline*}
where the first map is given by the identity times the inclusion of
$0$ in $(E(W)^\perp)^*$ and the second map is the projection.
\begin{NB2}
The morphisms
\(
      % T^* (T^*_Y X) \xleftarrow{\lsp{t}p'} p^{-1} T^*Y 
      % \xrightarrow{p_\pi} T^*Y
  T^* (\dot{T}^*_Y X) \xleftarrow{\lsp{t}{\dot{p}}'} \dot{p}^{-1} T^*Y 
  \xrightarrow{\dot{p}_\pi} T^*Y
\)
are identified with
\begin{multline*}
  E(W)\times E(W)^*\times \dot{E}(W)^\perp \times (E(W)^\perp)^*
\\
  \leftarrow
  E(W)\times E(W)^* \times \dot{E}(W)^\perp
  \rightarrow
  E(W)\times E(W)^*,
\end{multline*}
where the first map is given by the identity times the inclusion of
$0$ in $(E(W)^\perp)^*$ and the second map is the projection.
\end{NB2}
\end{NB}

\subsection{Proof}

We first study the behavior of the singular support under the functor
$\iota^*$. Let $L\in\mathcal Q_V$. By \cite[Cor.~6.4.4]{KaSha} we have
\begin{equation*}
  SS(\iota^* L) \subset \iota^\#(SS(L)).
\end{equation*}
In our situation, we have
$\iota^\#(SS(L)) = \lsp{t}\iota'(\iota_\pi^{-1}(SS(L)))$ by \eqref{eq:is}.

\begin{NB}
We claim that $\iota$ is non-characteristic for $L$ on
$\kappa^{-1}(T^*(\HomE_T\times\HomE_W))$, where 
$\kappa^{-1}(T^*(\HomE_T\times\HomE_W))$ is considered as a subset
of $T^*E(W)$ via $\lsp{t}\kappa'$.
In fact, $C_{\dot{T}^*_{E(W)}\HomE_V}(SS(L))$ is the intersection of
$SS(L)$ with the open subset $T^* (\dot{T}^*_{E(W)} \HomE_V)$ of $T^*
({T}^*_{E(W)} \HomE_V) \cong T^*\HomE_V$. In the matrix notation, the
open set is given by the condition $B'_{WT}\neq 0$ as explained
above. On the other hand, $(B,B')\in
\kappa^{-1}(T^*(\HomE_T\times\HomE_W))$ satisfies $B'_{TW}=0$ thanks
to \eqref{eq:kappa}. Therefore 
\(
  f_\infty^\#(SS(L)) =
  \dot{p}_\pi\lsp{t}{\dot{p}}^{\prime-1}(C_{\dot{T}^*_{E(W)}\HomE_V}(SS(L)))
\)

......

It seems that this is not true.
\end{NB}
 
Next study the functor $\kappa_!$.
Note that $\kappa\colon E(W)\to \HomE_T\times\HomE_W$ is the
projection of a vector bundle. Therefore the results in
\cite[\S5.5]{KaSha} are applicable. A complex $F$ in $\scr D(E(W))$ is
{\it conic\/} if $H^j(F)$ is locally constant on the orbits of the
$\R^+$-action for all $j$. In our situation, $F = i^*L$ satisfies
this condition. Then we have 
\begin{equation*}
  SS(\kappa_!(i^*L)) \subset 
  T^* (\HomE_T\times\HomE_W) \cap SS(i^*L)
  =
  \kappa_\pi\lsp{t}\kappa^{\prime-1} SS(i^*L).
\end{equation*}
See \cite[Prop.~5.5.4]{KaSha} for the first inclusion and
\cite[(5.5.11)]{KaSha} for the second equality. Combining two
estimates, we complete the proof of \thmref{thm:est}. The estimate
\eqref{eq:0sect} has been given during the proof.

\begin{NB}
We define a function $\varphi$ on $E(W)$ defined by
\begin{equation*}
  \varphi(B) = | B_{WT} |^2,
\end{equation*}
where $B'$ is the composite of $T \hookrightarrow V \xrightarrow{B}
V\to W$, and $|\ |$ is a norm on $\Hom(T,W)$.
Then $\varphi$ is a $C^\infty$-function such that
\begin{itemize}
\item $\kappa|_{\varphi\le t}$ is proper for all $t\in\R$, and
\item $-d\varphi\notin \left(SS(F) + \lsp{t}\kappa'(
    \kappa^{-1}(T^* (\HomE_T\times\HomE_W)) \right)$.
\end{itemize}
\end{NB}

\section{Conjectures}

\subsection{Quantum unipotent subgroup and singular supports}

\begin{NB}
  See note on 2011-02-16.
\end{NB}

Let $w$ be a Weyl group element as before.
Motivated by \subsecref{subsec:Cw}, we introduce a subset $\B'(w)$ in
$\B(\infty)$ by
\begin{equation*}
  \B'(w) \defeq \{ b\in\B(\infty)\mid
  SS(L_b)\cap \Lambda_V^w\neq \emptyset\},
\end{equation*}
where we suppose $L_b\in\mathcal P_V$ in the equation
$SS(L_b)\cap \Lambda_V^w\neq \emptyset$.
Equivalently $b\notin\B'(w)$ if and only if $SS(L_b)$ is contained in
the closed subvariety $\Lambda_V\setminus\Lambda_V^w$.

\begin{NB}
  Original definition:
\begin{equation*}
  \scr K^w(\mathcal Q_V) = \{ (L) \in \scr K(\mathcal Q_V)
  \mid SS(L) \cap \Lambda_V^w = \emptyset \}.
\end{equation*}
As $SS(L\oplus L') = SS(L)\cup SS(L')$ and $SS(L[1]) = SS(L)$ (see
\cite[Chap.~V]{KS}), it is an $\A$-submodule of $\scr K(\mathcal
Q_V)$.
\end{NB}

By \eqref{eq:SSest}, the condition $\Lambda_b\cap
\Lambda_V^w\neq\emptyset$ implies $b\in\B'(w)$.
Therefore $\B(w)\subset \B'(w)$ by \subsecref{subsec:Cw}.

Let $b\notin\B'(w)$. We have
\begin{equation*}
  SS(\Res(L_b)) \cap (\Lambda_T^w\times \Lambda_W^w)
  \subset SS(L_b)\cap \Lambda_V^w = \emptyset
\end{equation*}
by \eqref{eq:0sect} and the fact that $\mathcal C_w$ is an additive
category.
Writing 
\[
   \Res(L_b) = \bigoplus \left(L_{b_1}\boxtimes L_{b_2}
  \right)[n]^{\oplus
  r^{b_1,b_2}_{b;n}},
\]
we get
\begin{equation*}
  SS(L_{b_1}\boxtimes L_{b_2})\cap (\Lambda_T^w\times \Lambda_W^w)
  = \emptyset
\end{equation*}
if $r^{b_1,b_2}_{b;n}\neq 0$ for some $n$.
This is because $SS(L\oplus L') = SS(L)\cup SS(L')$ and $SS(L[1]) =
SS(L)$ (see \cite[Chap.~V]{KS}).
In the notation in \subsecref{subsec:dual} we have
$r^{b_1,b_2}_b = \sum_n r^{b_1,b_2}_{b;n} q^n$.

We have an estimate $SS(L_{b_1}\boxtimes L_{b_2})\subset
SS(L_{b_1})\times SS(L_{b_2})$ \cite[Prop.~5.4.1]{KS}.
However this does not imply $SS(L_{b_1})\times
SS(L_{b_2})\cap(\Lambda_T^w\times \Lambda_W^w) = \emptyset$, so we
need a finer estimate.
Since $L_{b_a}$ ($a=1,2$) is a perverse sheaf, it corresponds to a
regular holonomic $D$-module under the Riemann-Hilbert correspondence
(see e.g., \cite[Th.~7.2.5]{HTT}).
Then the singular support of $L_{b_a}$ is the same as the
characteristic variety of the corresponding $D$-module
\cite[Th.~11.3.3]{KaSha}, \cite[Th.~4.4.5]{HTT}.
\begin{NB}
  More precisely $\operatorname{Ch}(M_{b_a})
  = SS(\operatorname{Sol}_X(M_{b_a})
  = SS(DR(\mathbb D_X M_{b_a})[-\dim X])
  = SS(\mathbf D_X(DR(M_{b_a})))
  = SS(DR(M_{b_a}))^a$
\end{NB}%
As the characteristic variety of the exterior product is the
product of the characteristic varieties \cite[(11.2.22)]{KaSha},
we deduce $SS(L_{b_1}\boxtimes L_{b_2}) = SS(L_{b_1})\times SS(L_{b_2})$.
Therefore we have
\begin{equation*}
  (SS(L_{b_1})\cap \Lambda^w_T) \times
  (SS(L_{b_2})\cap \Lambda^w_W) = \emptyset.
\end{equation*}
Therefore either $b_1\notin\B'(w)$ or $b_2\notin\B'(w)$.
In other words, $b_1, b_2\in\B'(w)$ and $r^{b_1,b_2}_b\neq 0$ implies
$b\in\B'(w)$. Therefore $\bigoplus_{b\in
  \B'(w)}\Q(q) b^\up$ is a subalgebra of $\Un$ by \subsecref{subsec:dual}.

Our first conjecture is the following.
\begin{Conjecture}\label{conj:1}
  $\B'(w) = \B(w)$. In other words, if $b\notin\B(w)$, then
$SS(L_b)\subset \Lambda_V\setminus\Lambda_V^w$.
\end{Conjecture}

This conjecture is also equivalent to say $\bigoplus_{b\in
  \B'(w)}\Q(q) b^\up = \Un(w)$.

\subsection{Cluster algebra and singular supports}

Recall that Geiss-Leclerc-Schr\"oer \cite{GLSq} have introduced the
structure of a quantum cluster algebra on $\Un(w)$ and conjectured
that quantum cluster monomials are contained in $\B^\up(w)$.
If this is true, we should have two formulas \eqref{eq:two_formulas}
for dual canonical base elements corresponding to $Y_R$, $Y_{T_k}$, etc.
Conversely \eqref{eq:two_formulas} implies that $Y_R$, $Y_{T_k^*}$ are
dual canonical base elements by induction on the number of mutations.

Let us speculate why these formulas hold in terms of the corresponding
perverse sheaves.

The proposal here is the following conjecture:
\begin{Conjecture}\label{conj:2}
  Let $T$ be a reachable $\mathcal C_w$-maximal rigid module and
  $R\in\operatorname{add}(T)$. Let $\Lambda_{R}$ be the closure of the
  orbit through $R$ and $b_R$ the corresponding canonical base
  element.% Let $L_{b_R}$ be the corresponding perverse sheaf.

  If another canonical base element $b\in\B(w)$ satisfies
  \begin{equation*}
    \Lambda_{R}\subset SS(L_b),
  \end{equation*}
  we should have $b = b_R$.
\end{Conjecture}

If Conj.~\ref{conj:1} is true, $b\in\B(\infty)$ with
$\Lambda_{b_R}\subset SS(L_b)$ is contained in $\B(w)$. Therefore the
above conjecture holds for any $b\in\B(\infty)$.

This conjecture is true for a special case when $\Lambda_{R}$ is the
zero section $\HomE_V$ of $T^*\HomE_V$ and $G_V$ has an open orbit in
$\HomE_V$.
In fact, if $SS(L_{b})\supset \HomE_V$, we have
$\operatorname{supp}(L_{b}) = \HomE_V$. Then $L_{b}$ is
$G_V$-equivariant and gives an irreducible $G_V$-equivariant local
system on the open orbit in $\HomE_V$. As the stabilizer of a point is
connected from a general property from quiver representations, it must
be the trivial rank 1 local system. Thus $L_{b}$ is the constant sheaf
on $\HomE_V$.
\begin{NB}
  I use Cor.~8.2.6 of Hotta-Takeuchi-Tanisaki: $L_{b}$ is the IC
  extension of a local system $L$ on a $G_V$-invariant open subvariety
  $U'$ of $\HomE_V$. Then $U'$ contains the open orbit $U$. Then
  Cor.~8.2.6 says $L$ is the minimal extension of $L|U$, which is
  trivial. The constant sheaf on $U'$ is also the minimal extension
  of the trivial sheaf again by Cor.~8.2.6. Therefore $L$ is trivial.
\end{NB}%
In fact, the observation that $\operatorname{supp}(L_{b}) = \HomE_V$
implies $L_{b} = \text{the constant sheaf}$ was used in a crucial way
to prove the cluster character formula in \cite{cluster}.

If $SS(L_b)$ is irreducible for all $b$, Conj.~\ref{conj:2} is
obviously true. This condition is satisfied for $\g$ of type $A_4$, but
not for $A_5$ \cite{KS}.

Let us remark a relation between the above conjecture and a conjecture
in \cite[\S1.5]{GLS1}. This is pointed out by the referee to the author.
Let us define the {\it semicanonical base\/} $\{ f_Y\}$ of $\mathbf
U(\mathfrak n)$ as the dual base of the dual semicanonical base $\{
\rho_Y\}$. In \cite[\S1.5]{GLS1} it is conjectured that the
specialization of $b$ is a linear combination $\sum m_Y f_Y$
($m_Y\in\Z$), where the summation runs over irreducible components $Y$
of $SS(L_b)$.  (More precisely it is probably given by the
characteristic cycle (see \cite[2.2.2]{HTT} for the definition) of
$L_b$.)
Dually, an irreducible component $Y = \Lambda_b$ cannot be contained in 
other $SS(L_{b'})$ ($b'\neq b$) if $\left.b^\up\right|_{q=1}=\rho_{\Lambda_b}$.
\begin{NB}
  Suppose $b' = \sum m_{Y'} f_{Y'}$. Then
  $0 = \langle b', b^\up\rangle = \langle b', \rho_{\Lambda_b}\rangle
  = m_{\Lambda_b}$.
\end{NB}%
This is nothing but our conjecture. Thus under the conjecture in
\cite[\S1.5]{GLS1}, our conjecture is equivalent to
Conj.~\ref{conj:qmon} for reachable rigid modules.

Let us explain how the first formula in \eqref{eq:two_formulas} is
related to Conj.~\ref{conj:2}. We assume $R = T_1\oplus T_2$ for
brevity.
From the assumption $\Lambda_{R}$ contains the product
$\Lambda^\circ_{T_1}\times \Lambda^\circ_{T_2}$ as an open dense
subset. Here $\Lambda^\circ_{T_i}$ denote the open orbit through
$T_i$. Its closure is $\Lambda_{T_i}$.
Suppose that $b^\up$ appears in the product $b^\up_{T_1} b^\up_{T_2}$. Then
$r_{b}^{b_{T_1},b_{T_2}}\neq 0$. We have
\begin{equation*}
  \begin{split}
  SS(L_b) \cap (\Lambda^\circ_{T_1}\times \Lambda^\circ_{T_2})
  & \supset SS(\Res(L_b))\cap (\Lambda^\circ_{T_1}\times \Lambda^\circ_{T_2})
\\
  & \supset SS(L_{b_{T_1}}\boxtimes L_{b_{T_2}})\cap
  (\Lambda^\circ_{T_1}\times \Lambda^\circ_{T_2})
\\
  & = (SS(L_{b_{T_1}})\cap \Lambda^\circ_{T_1})\times 
  (SS(L_{b_{T_2}})\cap\Lambda^\circ_{T_2}),
  \end{split}
\end{equation*}
where the first inclusion is by \eqref{eq:0sect}, the second as a
shift of $L_{b_{T_1}}\boxtimes L_{b_{T_2}}$ is a direct summand of
$\Res{L_b}$, and the third equality was observed above. The last
expression is nonempty thanks to \eqref{eq:SSest}.
Therefore we have $SS(L_b)\supset \Lambda_R$. Then
Conj.~\ref{conj:2} implies that $b = b_R$. Therefore
$b^\up_{T_1} b^\up_{T_2}$ is a multiple of $b^\up_R$.
A refinement of this argument probably proves that
$b^\up_{T_1} b^\up_{T_2}$ is equal to $b^\up_R$ up to a power of $q$.

Let us turn to the second formula in \eqref{eq:two_formulas}. The same
argument above implies that
\begin{equation*}
  SS(L_b)\cap(\Lambda^\circ_{T_k^*}\times\Lambda^\circ_{T_k})\neq\emptyset
\end{equation*}
if $r^{b_{T_k^*},b_{T_k}}_b\neq 0$. From what we have explained in
\subsecref{subsec:cluster}, there are two irreducible components
$\Lambda_{T'}$, $\Lambda_{T''}$, where $T'$ and $T''$ are
$\Lambda$-modules given by non-trivial extensions of $T_k^*$ and
$T_k$. As a non-trivial extension can degenerate to the trivial one,
both $\Lambda_{T'}$ and $\Lambda_{T''}$ contain
$\Lambda^\circ_{T_k^*}\times\Lambda^\circ_{T_k}$. It is also easy to
check that $\dim \Lambda^\circ_{T_k^*}\times\Lambda^\circ_{T_k} = \dim
\Lambda_V - 1$, where $V$ is the underlying vector space of $T_k\oplus
T_k^*$.

\begin{Lemma}
  If an irreducible component $Y$ of $\Lambda_V$ contains
  $\Lambda^\circ_{T_k^*}\times\Lambda^\circ_{T_k}$, we have either $Y
  = \Lambda_{T'}$ or $=\Lambda_{T''}$.
\end{Lemma}

\begin{proof}
  Take a sequence $Z_n$ of points of $Y$ converging to the module
  $T_k^*\oplus T_k$, regarded as a point of $Y$. We may assume
  $Z_n\not\cong T_k^*\oplus T_k$.
  \begin{NB}
  Then we have
  \begin{equation*}
    \begin{split}
      & \dim \Ext^1(T/T_k, Z_n) \le \dim \Ext^1(T/T_k, T_k^*\oplus T_k) = 0
\\
      & \dim \Ext^1(Z_n, T/T_k) \le \dim \Ext^1(T_k^*\oplus T_k,T/T_k) = 0
    \end{split}
  \end{equation*}
  for sufficiently large $n$ by the upper semicontinuity of the
  dimension of cohomology groups. 
  \end{NB}%

  Then we have %We also have
  $\dim \Hom(Z_n, T_k^*\oplus T_k) \le \dim \Hom(T_k^*\oplus T_k,
  T_k^*\oplus T_k)$ for sufficiently large $n$ by the upper
  semicontinuity of the dimension of cohomology groups. If the
  equality holds, we can take $\xi_n\in \Hom(Z_n, T_k^*\oplus T_k)$
  converging to the identity of $\Hom(T_k^*\oplus T_k, T_k^*\oplus
  T_k)$ for $n\to\infty$. In particular, $\xi_n$ is invertible, hence
  $Z_n\cong T_k^*\oplus T_k$. This contradicts with our
  assumption. Therefore we have the strict inequality $\dim \Hom(Z_n,
  T_k^*\oplus T_k) < \dim \Hom(T_k^*\oplus T_k, T_k^*\oplus T_k)$. The
  same argument gives $\dim \Hom(T_k^*\oplus T_k,Z_n) < \dim
  \Hom(T_k^*\oplus T_k, T_k^*\oplus T_k)$.

  Therefore
  \begin{equation*}
    \begin{split}
    & \dim \Ext^1(T_k^*\oplus T_k,Z_n)
\\
    =\; & - (\dim V,\dim V) + \dim \Hom(T_k^*\oplus T_k, Z_n)
    + \dim \Hom(Z_n,T_k^*\oplus T_k)
\\
    \le\; &
    - (\dim V,\dim V) + 2 \dim \Hom(T_k^*\oplus T_k, T_k^*\oplus T_k) - 2
\\
   =\; &
    \Ext^1(T_k^*\oplus T_k, T_k^*\oplus T_k) - 2 = 0,
    \end{split}
  \end{equation*}
  where we have used the formula in \cite[Lem.~2.1]{GLSc}.  The upper
  semicontinuity also shows $\Ext^1(T/T_k, Z_n) = 0$, hence
  $Z_n\in\operatorname{add}(T/T_k)$.

  As the inequality above must be an equality, we get
  \begin{equation*}
    \begin{split}
    & \dim \Hom(T_k^*\oplus T_k,Z_n) = \dim \Hom(Z_n,T_k^*\oplus T_k) \\
    = \; & \dim \Hom(T_k^*\oplus T_k,T_k^*\oplus T_k) - 1.
    \end{split}
  \end{equation*}
Therefore
\begin{equation*}
  \begin{split}
  & \dim \Hom(T_k^*, Z_n) + \dim \Hom(T_k,Z_n)\\
  =\; &
  \dim \Hom(T_k^*, T_k^*\oplus T_k) + \dim \Hom(T_k,T_k^*\oplus T_k) - 1.
  \end{split}
\end{equation*}
Note that $\dim \Hom(T_k^*, Z_n) \le \dim \Hom(T_k^*, T_k^*\oplus
T_k)$ and $\dim \Hom(T_k, Z_n) \le \dim \Hom(T_k, T_k^*\oplus T_k)$
by the semicontinuity.
Hence the above implies that one of inequalities must be an equality.
Suppose that the first one is an equality. Then we have 
\begin{equation*}
  \begin{gathered}
    \dim \Hom(T_k^*, Z_n) = \dim \Hom(T_k^*, T_k^*\oplus T_k),
\\
    \dim \Hom(T_k, Z_n) = \dim \Hom(T_k,T_k^*\oplus T_k) - 1.
  \end{gathered}
\end{equation*}
The same argument shows that $\dim \Hom(Z_n, T_k^*) = \dim
\Hom(T_k^*\oplus T_k, T_k^*)$ or $\dim \Hom(Z_n, T_k) = \dim
\Hom(T_k^*\oplus T_k, T_k)$.
The first equality is impossible, as 
$0 = \dim \Ext^1(T_k^*,Z_n) \neq \dim \Ext^1(T_k^*,T_k\oplus T_k) = 1$ and
the above dimension formula. Therefore we have
\begin{equation*}
  \begin{gathered}
    \dim \Hom(Z_n, T_k^*) = \dim \Hom(T_k^*\oplus T_k,T_k^*) - 1,
\\
    \dim \Hom(Z_n, T_k) = \dim \Hom(T_k^*\oplus T_k,T_k).
  \end{gathered}
\end{equation*}

We take $\eta_n\in\Hom(T_k^*,Z_n)$ converging to $\id_{T_k^*}\oplus 0$
in $\Hom(T_k^*,T_k^*\oplus T_k)$. In particular, $\eta_n$ is injective
for sufficiently large $n$. We consider an exact sequence
\begin{equation*}
  0 \to \Hom(Z_n/\Ima\eta_n, T_k) \to \Hom(Z_n, T_k) \to \Hom(\Ima\eta_n, T_k).
\end{equation*}
The next term $\Ext^1(Z_n/\Ima\eta_n, T_k)$ vanishes, as we have $\dim
\Ext^1(Z_n/\Ima\eta_n,T_k) \le \dim\Ext^1(T_k,T_k) = 0$ by the upper
semicontinuity. Therefore we have
\begin{equation*}
  \dim \Hom(Z_n/\Ima\eta_n, T_k) = \dim \Hom(T_k, T_k).
\end{equation*}
We take $\zeta_n\in\Hom(Z_n/\Ima\eta_n,T_k)$ converging to
$\id_{T_k}$. Then $\zeta_n$ is an isomorphism for sufficiently large
$n$. Composing the projection $p\colon Z_n\to Z_n/\Ima\eta_n$ with
$\zeta_n$, we have an exact sequence
\begin{equation*}
  0\to T_k^* \xrightarrow{\eta_n} Z_n \xrightarrow{\zeta_n\circ p}
  T_k \to 0.
\end{equation*}
This shows that $Z_n\cong T''$.

When $\dim \Hom(T_k,Z_n) = \dim \Hom(T_k,T_k^*\oplus T_k)$, the same
argument shows that $Z_n\cong T'$.
\end{proof}

Now Conj.~\ref{conj:2} implies that $b_{T_k^*}^\up b_{T_k}^\up$ is a
linear combination of $b_{T'}^\up$ and $b_{T''}^\up$. A refinement of
the argument hopefully gives the second formula in \eqref{eq:two_formulas}.

\bibliographystyle{myamsplain}
\bibliography{mybib,nakajima,qGLS}
\end{document}